\documentclass[12pt]{article}

\usepackage{graphicx,url}
\usepackage{amsmath,amsthm,amssymb}
\usepackage[noadjust,sort]{cite}

\usepackage[colorlinks=true,citecolor=black,linkcolor=black,urlcolor=blue]{hyperref}

\allowdisplaybreaks

\makeatletter
\g@addto@macro\normalsize{%
  \setlength\abovedisplayskip{8pt plus 3pt minus 3pt}
  \setlength\belowdisplayskip{8pt plus 3pt minus 3pt}
  \setlength\abovedisplayshortskip{6pt plus 3pt minus 2pt}
  \setlength\belowdisplayshortskip{6pt plus 3pt minus 2pt}
}
\makeatother

\date{}

\normalsize

\numberwithin{equation}{section}

\newcommand\calL{\mathcal{L}}
\renewcommand\dfrac[2]{\lower0.15ex\hbox{\large$\frac{#1}{#2}$}}
\newcommand\nicebreak{\vskip 0pt plus 50pt\penalty-300\vskip 0pt plus -50pt }
\newcommand\eps{\varepsilon}

\newtheorem{thm}{Theorem}[section]
\newtheorem{cor}[thm]{Corollary}

\newtheorem{lemma}[thm]{Lemma}
\newtheorem{conj}[thm]{Conjecture}

\def\dfrac#1#2{\lower0.15ex\hbox{\large$\textstyle\frac{#1}{#2}$}}
\def\({\bigl(}
\def\){\bigr)}

\def\({\bigl(}
\def\){\bigr)}

\let\eps=\varepsilon

\begin{document}

\title{Factorisation of the complete graph into spanning regular factors}

\author{
Mahdieh Hasheminezhad\\
\small Department of Mathematical Sciences\\[-0.8ex]
\small Yazd University\\[-0.8ex]
\small Yazd, Iran\\
\small\tt hasheminezhad@yazd.ac.ir
\and
Brendan D. McKay\thanks{Supported by the Australian Research Council}\\
\small School of Computing\\[-0.8ex]
\small Australian National University\\[-0.8ex]
\small Canberra, ACT 2601, Australia\\
\small\tt brendan.mckay@anu.edu.au
}

\maketitle

\begin{abstract}
We enumerate factorisations of the complete graph into spanning regular
graphs in several cases, including when the degrees of all the factors
except for one or two are small.  The resulting asymptotic behaviour is 
seen to generalise the number of regular graphs in a simple way.
This leads us to conjecture a general formula when the number of
factors is vanishing compared to the number of vertices.
\end{abstract}

\nicebreak

\section{Introduction}

A classical problem in enumerative graph theory is the asymptotic
number of regular graphs.
This has now been solved by the overlap of three papers~\cite{LW,MW1,MW2},
with the interesting conclusion that the same formula holds in the sparse
and dense regimes.
\begin{thm}\label{thm:reg}
   Let $R_d(n)$ be the number of regular graphs of degree $d$ and
   order~$n$, where $1\le d\le n-2$.  Define
   $\lambda=d/(n-1)$.  Then, as $n\to\infty$,
   \[ 
     R_d(n) \sim \bigl( \lambda^{\lambda} (1-\lambda)^{1-\lambda} \bigr)
         ^{\!\binom{n}{2}} \binom{n-1}{d}^{\!n} e^{1/2}\, 2^{1/2}.
   \]
\end{thm}
In the above, and throughout the paper, we tacitly assume that regular
graphs have even degree whenever the number of vertices is odd.

We can consider $R_d(n)$ to count the number partitions of the edges of~$K_n$
into two spanning regular subgraphs, one of degree $d$ and one of degree~$n-d-1$.
This suggests a generalization: how many ways are there to partition the edges
of~$K_n$ into more than two spanning regular subgraphs, of specified degrees?

For integers $d_0,d_1,\ldots,d_k\ge 1$ such that $\sum_{i=0}^k d_i=n-1$,
define $R(n; d_0,\ldots,d_k)$ to be the number of ways to partition the edges of
$K_n$ into spanning regular subgraphs of degree $d_0,d_1,\ldots,d_k$?
We conjecture that for $k=o(n)$, the asymptotic answer is a simple
generalisation of Theorem~\ref{thm:reg}.
\begin{conj}\label{conj:main}
  Define $\lambda_i=d_i/(n-1)$ for $0\le i\le k$.
  If $k=o(n)$, then $R(n; d_0,\ldots,d_k)\sim R'(n; d_0,\ldots,d_k)$, where
   \[
     R'(n;d_0,\ldots,d_k) = \biggl( \prod_{i=0}^k \lambda_i^{\lambda_i}\biggr)^{\!\binom{n}{2}}
    \binom{n-1}{d_0,\ldots,d_k}^{\!n} e^{k/4}\, 2^{k/2} ,
   \]
 using a multinomial coefficient.
\end{conj}

We will prove the conjecture in four cases.
\begin{enumerate}\itemsep=0pt
  \item $k=1$ and $1\le d_1\le n-2$.
  \item $\sum_{i=1}^k d_i=o(n^{1/3})$ and $\sum_{1\le i\le j<t\le k} d_id_jd_t^2=o(n)$.
  \item $k=o(n^{5/6})$ and $d_1=\cdots=d_k=1$.
  \item $\min\{d_1,n-1-d_1\} \ge cn/\log n$ for some constant $c>\frac23$
        and $\sum_{i=2}^k d_i = O(n^\eps)$ for sufficiently small $\eps>0$.
\end{enumerate}

Case~1 is just a restatement of Theorem~\ref{thm:reg}, since
$R_d(n)=R(n;n{-}1{-}d,d)$.
Case~2 will follow from a switching argument applied to the probability of two
randomly labelled regular graphs being edge-disjoint.
Case~3 is a consequence of~\cite{McL}.
Case~4 will follow from a combination of~\cite{ranx}, \cite{MW2} and Part~1.
Case~2 with $k=2$, $d_1=2$ and $d_2=O(1)$ appears in~\cite{HR}.

\section{Two regular graphs}\label{s:stepone}

We begin by considering the case of two arbitrary regular graphs
which are randomly labelled.  What is the probability of them being
edge-disjoint?

\begin{lemma}\label{lem:basic}
Let $D$ and $H$ be regular graphs on $n$ vertices, where $D$ is $d$-regular 
for $d\ge 1$ and $H$ is $h$-regular for $h\ge 1$.
Suppose $d^2h^2=o(n)$ and define $M = \lceil \max\{8dh,\log n\}\rceil$, 
Then, with probability $1- O(\frac{d^2 h^2}{n})$, a random relabeling 
of the vertices of $H$ does not have any common path of length 2 with $D$
and the number of common edges with $D$ is less than~$M$.
\end{lemma}

\begin{proof}

The probability that a given path of length 2 of $H$ is mapped to a path of length~2
of $D$ is $\frac{2n\binom{d}{2}(n-3)!}{n!}$ and
the number of paths of length 2 of $H$ is $n\binom{h}{2}$.
So the expected number of paths of length 2 which are mapped to a path of length 2 is at most 
\[
  \frac{2n^2\binom{d}{2}\binom{h}{2}(n-3)!}{n!}=O\Bigl(\frac{d^2 h^2}{n}\Bigr).
\]
So with probability $1- O(\frac{d^2 h^2}{n})$, a random relabeling of the vertices of $H$ does not have
any common path of length 2 with $D$.     

Considering  a given set  $S$ of $M$ edges of $H$, the probability that a random relabeling of the vertices of $H$  maps $S$ to $M$ distinct edges of $D$ is
\[ 
  \frac{\binom{nd/2}{M} M!\, 2^M(n-2M)!}{n!}
\]
Since there are $\binom{nh/2}{M}$ sets of $M$ different edges of $H$,  the expected number of such relabeling is at most
\[
  E=\frac{\binom{nh/2}{M}\binom{nd/2}{M}M!\,2^M(n-2M)!}{n!} \le \frac{n^{2M}e^Md^Mh^M}{2^MM^M(n-2M)^{2M} }.
\] 
For big enough $n$, $\frac{n^2}{(n-2M)^2} \le 2 $.
Hence, the expected number of such relabeling is at most $O((\frac{edh}{M})^M)$.
Since $M\ge 8dh \ge e^2dh$ and $M\ge \log n$, $(\frac{edh}{M})^M\le e^{-M}\le n^{-1}$.
This is dominated by the term $\frac{d^2h^2}{n}$ from the first part of the proof, so
 with probability  $1- O(\frac{d^2 h^2}{n})$  a random relabeling 
of the vertices of $H$ does not have any common path of length 2 with $D$ and the number of common edges with $D$ is less than $M$.
\end{proof}

Let $\calL(t)$ be the set  of all relabeling of the vertices of $H$ with no common paths of length 2 with $D$ and exactly $t$ common distinct edges with $D$.
     Define $L(t)=|\calL(t)|$, so in particular
    the number of relabeling of the vertices of $H$ with no common edges with $D$ is $L(0)$. 
    Let
    \[
        T = \sum_{t =0}^{M-1}
                L(t)
    \]
   
    In the next
   step we estimate the value of $T/L(0)$ by the switching method.

\section{The switching}\label{s:switching}

A \textit{forward switching} is a permutation $(a \,  e)(b \, f)$ of the vertices of $H$ such that
\begin{itemize}\itemsep=0pt
 \item vertices $a$, $e$, $b$, $f$ are all distinct.
 \item  $ab$ is a common edge of $D$ and $H$,
 \item  $ef$ is non-edge of $D$, and 
 \item after the permutation, the common edges of $D$ and $H$ are the
         same except that $ab$ is no longer a common edge.
\end{itemize}

\noindent
A \textit{reverse switching} is a permutation $(a \,  e)(b \, f)$ of the vertices
      of $H$ such that
\begin{itemize}\itemsep=0pt
 \item vertices $a$, $e$, $b$, $f$ are all distinct.
 \item  $ab$ is an edge of $H$ that is not an edge of $D$,
 \item  $ef$  is an edge of $D$ that is not an edge of $H$, and 
 \item after the permutation, the common edges of $D$ and $H$ are the
         same except that $ef$ is a common edge of $D$ and $H$.
\end{itemize}

\begin{lemma}\label{lem:Lrat}
Assume $d,h\ge 1$, $d^2h^2=o(n)$, and define
$m_d=\frac{nd}{2}$ and $m_h=\frac{nh}{2}$.
Then for $1\le t\le M$ and $n\to\infty$ we have uniformly
\[
  \frac{L(t)}{L(t-1)}=\frac{(m_d-(t-1)) (m_h-(t-1))}{t\,\binom{n}{2}}
 \Bigl(1+O\Bigl(\frac{dh}{n}\Bigr)\Bigr).
\]
\end{lemma}
\begin{proof}
By using a forward switching, we will convert  a relabeling $\pi \in \calL(t)$ to a relabeling $\pi'\in \calL(t-1)$.
Without lose of generality,  we suppose that $\pi=(1)$,  since our estimates
will be independent of the structure of $H$ other than its degree.

There are $t$ choices for edge $ab$ and at most 
$2(\binom{n}{2}-m_d)=2\binom{n}{2}(1+O(d/n))$ choices for $e$ and $f$.
But some of the choices of $e$ and $f$ are not suitable. There are at most $4n$  choices of $e$ and $f$ such that $a$, $e$, $b$, $f$ are not all distinct.

 Since  $D$ and $H$   have  no paths of length 2 in common, there are no other  common edges of $H$ and $D$ incident to $a$ or $b$.
If neither $e$ nor $f$ are an end vertex of a common edge, then no common edge is destroyed.  Therefore, 
there are at most $4tn$ pairs $\{e,f\}$ that can destroy a common edge.

If none of the following happens, no new common edge is created by the forward switching:
\begin{itemize}\itemsep=0pt
\item vertex $e$  has  a neighbor in $H$ which is  a neighbor of $a$ in $D$.
\item vertex $e$  has  a neighbor in $D$ which is  a neighbor of $a$ in $H$.
\item vertex $f$  has  a neighbor in $H$ which is  a neighbor of $b$ in $D$.
\item vertex $f$  has  a neighbor in $D$ which is  a neighbor of $b$ in $H$.
\end{itemize}

So, there are at most $4dh$  vertices which do not satisfy one of the above conditions. 
So there are at most $4dh(n-2(t-1))$  unsuitable couples which can produces some new common edges. Therefore the number of forward switchings is
\[
  W_F=t\binom{n}{2} \Bigl(1+O\Bigl(\frac{dh}{n} \Bigr)\Bigr).
\]

A reverse switching converts a relabeling $\pi' \in \calL(t-1)$ to a relabeling $\pi$ in $\calL(t)$.
Again, without loss of generality, we can suppose that $\pi'=(1)$.
There are 
$(m_h-t+1)$ choices for edge $ab$ and there is at most $2(m_d-t+1)$ choices for $e$ and~$f$. But some of the choices of $e$ and $f$ are not suitable.

 There are at most $4n$  choices of $e$ and $f$ such that $a$, $e$, $b$, $f$ are not all distinct.
 If neither $e$ nor $f$ are an end vertex of a common edge, then no common edge  is destroyed and also no common edge except $ef$  has an end vertex in $e$ or $f$.
 So, there are at most $2td$ unsuitable choices for $e$ and $f$ that may can destroy a common edge or construct a common path of length 2. 

If none of the following happens, no other new common edge obtains after doing the reverse switching:

\begin{itemize}\itemsep=0pt
\item vertex $e$  has  a neighbor in $H$ which is  a neighbor of $a$ in $D$.
\item vertex $e$  has  a neighbor in $D$ which is  a neighbor of $a$ in $H$.
\item vertex $f$  has  a neighbor in $H$ which is  a neighbor of $b$ in $D$.
\item vertex $f$  has  a neighbor in $D$ which is  a neighbor of $b$ in $H$.
\end{itemize}

So, there are at most $4dh$ vertices which does not satisfy one of the above conditions.
So there are at most $4d^2h$  unsuitable couples which can produces some other new common edges. 

 Therefore the number of reverse switchings is
\[
  W_R= \bigl(m_d-t+1\bigr) \bigl(m_h-t+1\bigr) \Bigl(1+O\Bigl(\frac{dh}{n}\Bigr)\Bigr).
\]

By considering the number of forward switchings  and  the number of reverse switchings, we have
\[
  \frac{L(t)}{L(t-1)}=\frac{(m_d-t+1) (m_h-t+1)}{{t \binom{n}{2}}}
 \Bigl(1+O\Bigl(\frac{dh}{n}\Bigr)\Bigr). \qedhere
\]
\end{proof}

We will need the following summation lemma from \cite[Cor.~4.5]{GreenhillMcKay2006}.

\begin{lemma}\label{sumlemma}
Let $Z \ge 2$ be an integer and, for $1\le i \le Z$, let real numbers $A(i)$, $B(i)$
be given such that $A(i) \ge 0$ and $1 - (i - 1)B(i) \ge 0$. Define
 $A_1 = \min _{i=1}^{Z} A(i)$, $A_2 =\max_{i=1}^{Z} A(i), 
 C_1 = \min_{i=1}^{Z} A(i)B(i)$ and $C_2 = \max_{i=1}^{Z} A(i)B(i)$.
 Suppose that there exists
$\hat{c}$ with $0 < \hat{c} < \frac13$
 such that $\max\{A/Z, |C|\} \le \hat{c}$
 for all $A \in [A_1,A_2]$,
$C \in [C_1,C_2]$. Define $n_0, \ldots , n_Z$ by $n_0 = 1$ and
\[
   n_i/ n_{i-1}=\dfrac1i A(i) (1 - (i -1)B(i)) 
\]
for $1 \le i \le Z$, with the following interpretation: if $A(i) = 0$ or $1 - (i - 1)B(i) = 0$, then
$n_j = 0$ for $ i \le j \le Z$. Then 
\[
   \Sigma_{1} \le \sum_{i=0}^{Z}n_i \le \Sigma_2,
\]
where
\[
  \Sigma_1 = \exp\bigl(A_1 - \dfrac12 A_1C_2)-(2e\hat{c}\bigr)^Z
\]
and 
\[
    \Sigma_2=\exp\bigl(A_2-\dfrac12A_2C_1+\dfrac12 A_2C_1^2\bigr)+ (2e\hat{c} )^Z. 
    \tag*{\qed}
\]
\end{lemma}

\begin{lemma}\label{lem:tsum}
Let $d,h\ge 1$ and $d^2h^2 = o(n)$.  Then, as $n\to\infty$,
\[
  \frac{T}{L(0)}=\exp\Bigl(\dfrac12 dh+O\Bigl(\frac{d^2h^2}{n}\Bigr)\Bigr).
\]
\end{lemma}

\begin{proof}
The condition $d^2h^2 = o(n)$ and the definition of $M$
allow us to write Lemma~\ref{lem:Lrat} as
\[
  \frac{L(t)}{L(t-1)} = 
  \frac{dh}{2t} \Bigl(1 - (t-1)\frac{2(d+h)}{d h n}\Bigr)
       \Bigl(1 + O\Bigl(\frac{dh}{n}\Bigr)\Bigr),
\]
where the implicit constant in the $O(\,)$ depends on $t$ but is
uniformly bounded for $1\le t\le M$. 
For $0 \le t \le M$, define
\begin{align*}
A(t) &= \dfrac12 dh \Bigl(1 + O\Bigl(\frac{dh}{n}\Bigr)\Bigr) \\
B(t) &= \frac{2(d+h)}{d h n},
\end{align*}

Then  $A(t) \ge 0$ and $1 - (t - 1)B(t) \ge 0$ as $n\to\infty$.
Define $A_1$, $A_2$, $C_1$ and $C_2$ as in Lemma~\ref{sumlemma}. 
This gives
\begin{align*}
A_1,A_2&=\dfrac12 dh\Bigl(1 + O\Bigl(\frac{dh}{n}\Bigr)\Bigr)
 = O(dh) , \\
C_1,C_2&= O\Bigl( \frac{d+h}{n}\Bigr).
\end{align*}

The condition $\max\{A/M, |C|\} \le \hat{c}$ of Lemma~\ref{sumlemma} is satisfied
with $\hat{c}=\frac1{15}$, due to $d^2h^2 =o(n)$ and the definition of~$M$.
We also have
\[ 
  A_1C_2, A_2C_1, A_2C_1^2 = O\Bigl(\frac{dh(d+h)}{n}\Bigr)
\]
and $(2e\hat{c})^M \le n^{\log(2e/15)} = o(n^{-1})$.
Therefore, by Lemma~\ref{sumlemma}, 
\[
\frac{T}{L(0)}=\exp\Bigl(\,\dfrac12 dh
  +O\Bigl(\frac{d^2 h^2}{n}\Bigr)\Bigr). \qedhere
\]
\end{proof}

\begin{thm}\label{thm:prob}
Let $D$ and $H$ be regular graphs on $n$ vertices, where $D$ is $d$-regular 
for $d\ge 1$ and $H$ is $h$-regular for $h\ge 1$.
Suppose $d^2h^2=o(n)$ as $n\to\infty$.  Then the probability that
a random relabelling of $H$ is edge-disjoint from $D$ is
\[
  \exp\Bigl(-\dfrac12 dh +O\Bigl(\frac{d^2 h^2}{n}\Bigr)\Bigr).
 \]
\end{thm}
\begin{proof}
The probability that there are no common paths of length two and less
than $M$ common edges is $1-O(d^2h^2/n)$ by Lemma~\ref{lem:basic}.
Subject to those conditions, the probability of no common edge is
\[
  \exp\Bigl(-\dfrac12 dh+O\Bigl(\frac{d^2 h^2}{n}\Bigr)\Bigr),
\]
by Lemma~\ref{lem:tsum}.  Multiplying these probabilities together
gives the theorem.
\end{proof}

Since the formula in Theorem~\ref{thm:prob} does not depend on the
structure of $D$ or $H$, the same formula holds if one or both $H$ and~$D$
are random regular graphs with the given degrees.

\begin{cor}\label{cor:disj}
Let $D_1,D_2,\ldots,D_k$ be regular graphs on $n$ vertices,
with degrees $d_1,\ldots,d_k$, respectively.
Assume $d_i \ge 1$,   $i=1, \ldots, k$ and $\sum_{1 \le i \le j < t\le k} d_i d_j d_t^2=o(n)$ as $n\to\infty$.
Then the probability that $D_1,D_2,\ldots,D_k$ are edge-disjoint after
random relabelling is
\[
   \exp\Bigl(-\dfrac12 \sum_{1 \le i < j \le k} d_id_j
  +O\Bigl(\frac{\sum_{1 \le i \le j < t\le k} d_i d_j d_t^2}{n}\Bigr)\Bigr).
\]
\end{cor}
\begin{proof}
First compute the probability that a random relabeling of $D_2$ has no common edges with $D_1$. 
Then, let $D_1+D_2$ be the $d_1+d_2$-regular graph obtained from merging $D_1$ and $D_2$ and compute the probability that a random relabeling of $D_3$ has no common edges with $D_1+D_2$.
We  continue doing these computations inductively. The corollary obtained by applying Theorem~\ref{thm:prob}  repeatedly.
It can be shown that this form of the error term is the best that follows from 
Theorem~\ref{thm:prob} and that it is minimized when $d_1\ge \cdots \ge d_k$.
\end{proof}

Now we can prove our first new case of Conjecture~\ref{conj:main}.

\begin{thm}\label{thm:coloured}
  Let $d_0,d_1,\ldots,d_k\ge 1$ be integers such that
  $\sum_{j=0}^n d_j=n-1$.  Define $d=\sum_{i=1}^k d_i$ and
  suppose that $\sum_{1 \le i \le j < t\le k} d_i d_j d_t^2=o(n)$
  and $d^3=o(n)$ as $n\to\infty$.
 Define $\lambda_j=\frac{d_j}{n-1}$ for each $j$. Then
\[
    R(n;d_0,\ldots,d_k) =
    R'(n;d_0,\ldots,d_k) \Bigl(1 + O\Bigl( \frac {d^3}{n} 
      + \frac{1}{n}\sum_{1 \le i \le j < t\le k} d_i d_j d_t^2 \Bigr)\Bigr),
\]
where
 $R'(n;d_0,\ldots,d_k)$ is defined in Conjecture~\ref{conj:main}.
\end{thm}
\begin{proof}
  We can construct all such factorisations by choosing edge-disjoint regular
  graphs $D_1,\ldots,D_k$ of degrees $d_1,\ldots,d_k$.

 In the case of $k=1$, we are counting regular graphs, and it was
 shown in~\cite{MW2} that
 \begin{equation}\label{eq:reg}
      R_{d_i}(n) = R'(n; n{-}1{-}d_i,d_i)\bigl( 1+ O(d_i^2/n) \bigr),
  \end{equation}
  provided the error term is $o(1)$.
  Now we can write
  \[
    R(n;d_0,\ldots,d_k) = P(n;d_1,\ldots,d_k)\, \prod_{i=1}^k \,R_{d_i}(n),
  \]
  where $P(n;d_1,\ldots,d_k)$ is the expression in Corollary~\ref{cor:disj}.
  Using $(1-x)\log (1-x)=-x+\frac12 x^2 + O(x^3)$ we have
  \begin{align*}
    \log\frac{\prod_{i=1}^k \bigl( \lambda_i^{\lambda_i}(1-\lambda_i)^{1-\lambda_i}\bigr)}
                 {\bigl(1-\sum_{i=1}^k \lambda_i\bigr)^{1-\sum_{i=1}^k \lambda_i}
                   \prod_{i=1}^k  \lambda_i^{\lambda_i}}
        &= -\sum_{1\le i<j\le k} \lambda_i\lambda_j + O(d^3/n^3) \\
        &= -\frac{1}{n(n-1)} \sum_{1\le i<j\le k} d_id_j + O(d^3/n^3).
 \end{align*}
 For $1\le x=o(n^{1/2})$ we have
 $\sum_{i=0}^{x-1}\log(1-i/(n-1)) = -x(x-1)/(2n) + O(x^3/n^2)$, so
\[
  \log\frac{ \binom{n-1}{d_0,\ldots,d_k}} {\prod_{i=1}^k \binom{n-1}{d_i}}
  = \frac{1}{n} \sum_{1\le i<j\le k} d_id_j + O(d^3/n^2).
\]
Putting the parts together with~\eqref{eq:reg} completes the proof.
\end{proof}


Conjecture~\ref{conj:main} proposes that $R'(n;d_0,d_1,\ldots,d_n)$
matches $R'(n;d_0,d_1,\ldots,d_n)$  over a much wider domain than we can proved.
One case we can test using earlier work concerns partial
1-factorisations of~$K_n$.
Let $F(n,k)$ be the number of \textit{sequences} of $k$ disjoint perfect
matchings in~$K_n$. Note that $F(n,n-2)=F(n,n-1)$; we will use only
$F(n,n-2)$.
In our notation, $F(n,k)=R(n;n{-}k{-}1,1,\ldots,1)$, where there are $k$
explicit ones.
In~\cite{McL}, McLeod found the asymptotic value of $F(n,k)$ for $k=o(n^{5/6})$,
namely,
\[
  F(n,k)
  \sim \Bigl( \frac{n!}{2^{n/2}(n/2)!}\Bigr)^{\!k}
     \Bigl( \frac{n!}{n^k(n-k)!}\Bigr)^{\!n/2}
     \Bigl(1 - \frac{k}{n}\Bigr)^{\!n/4} e^{k/4}.
\]
The reader can check that this expression is equal to
$R'(n;n{-}1{-}k,1,\ldots,1)$ asymptotically when $k=o(n^{6/7})$.

\begin{figure}[ht!]
   \[  \includegraphics[scale=0.44]{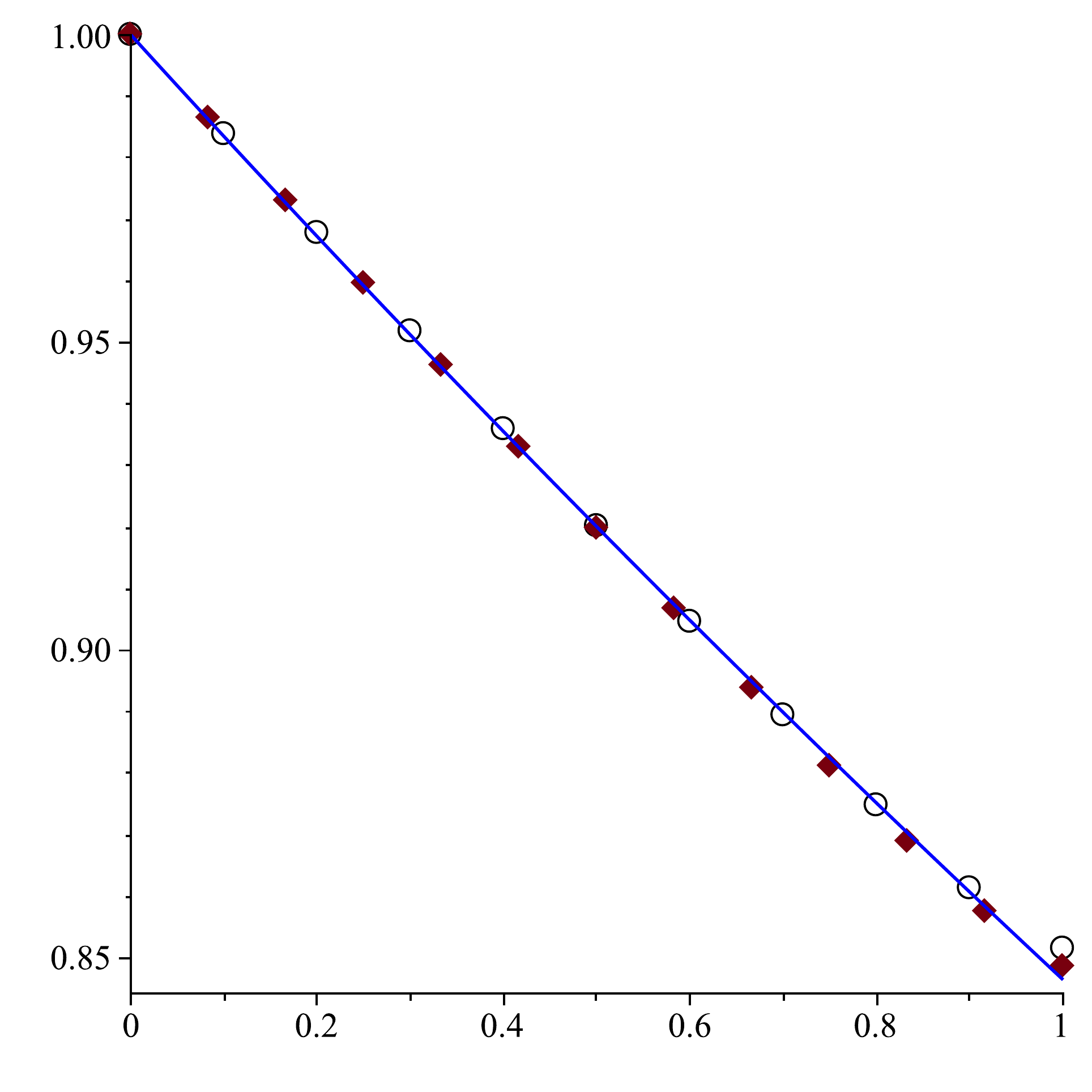} \]
   \vspace*{-6ex}
   \caption{Partial 1-factorisations.
   $F(n,k)/R'(n;n{-}1{-}k,1,\ldots,1)$ for $n=12$ (circles)
   and $n=14$ (diamonds). The horizontal scale is $x=k/(n-2)$
   and the curve is $e^{-x/6}$.\label{fig}}
\end{figure}

To illustrate what happens when $k$ is even larger, Figure~\ref{fig}
shows the ratio
\[
  F(n,k)/R'(n;n{-}1{-}k,1,\ldots,1)
\]
for the largest
sizes for which $F(n,k)$ is known exactly~\cite{DGM, KO}.                                                                                                                                                                                                                                                                                                                                                                                                                                                                                                                                                                                                                                                                                                                                                                                                                                                                                                                                                                                                                                                                                                                                                                                                                                                                                                                                                                                                                                                                                                                                                                                                                                                                                                                                                                                                                                                                                                                                                                                                                                                                                                                                                                                                                                                                                                            
Experiment suggests that $F(n,x(n-2))/R'(n;n{-}1{-}x(n-2),1,\ldots,1)$
converges to a continuous function $f(x)$ as $n\to\infty$ with
$x$ fixed. Conjecture~\ref{conj:main} in this cases corresponds
to $f(0)=1$.  At the other end, $x=1$ corresponding to complete
1-factorisations, we can also check the 3-digit estimates of
$F(n,n-2)$ in~\cite{DGM} for $n=16,18$---they give ratios 0.844
and 0.845, respectively.

Another case that we can solve is when there are two components
of high degree and some number of low degree.
In~\cite{ranx}, the second author considered the case when $k=2$,
$\min\{d_1,n-1-d_1\}\ge cn/\log n$ for some $c>\frac23$, 
and $d_2=o(n^{\eps})$ for some sufficiently small
$\eps>0$. In this case, the probability that a
random $d_1$-regular graph and an arbitrary $d_2$-regular
graph are edge-disjoint is asymptotic to
\begin{equation}\label{eq:join}
   (1-\lambda_1)^{d_2n/2}
     \exp\Bigl( -\frac{\lambda_1 d_2(d_2-2)}{4(1-\lambda_1)}\Bigr).
\end{equation}
Note that this is not uniform over $d_1$-regular graphs, but an
average over them.  However, within the error term it is uniform
over $d_2$-regular graphs and that is enough.

\begin{thm}\label{thm:ranx}
Let $d_1,d_2,\ldots,d_k\ge 1$ be such that
$\min\{d_1,n-1-d_1\}\ge cn/\log n$ for some $c>\frac23$
and $d_2+\cdots+d_k=O(n^\eps)$ for sufficiently small
$\eps>0$.  Let $d_0=n-1-\sum_{i=1}^k d_i$.
Then, as $n\to\infty$,
\[
  R(n; d_0,\ldots,d_k) \sim R'(n; d_0,\ldots,d_k).
\]
\end{thm}
\begin{proof}
 Let $\hat d=\sum_{i=2}^k d_i$.
 By Theorem~\ref{thm:coloured}, the number of
 $\hat d$-regular graphs partitioned into $d_2,\ldots,d_k$-regular
 graphs is asymptotic to $R'(n;n{-}1{-}\hat d,d_2,\ldots,d_k)$.
 By~\eqref{eq:reg}, the number of $d_1$-regular graphs is
 asymptotic to $R'(n;n{-}1{-}d_1,d_1)$. The probability of these
 two graphs being edge disjoint when the $d_1$-regular graph
 is chosen randomly is given by~\eqref{eq:join} (with $\hat d$
 in place of $d_2$).
 The product of these three quantities is asymptotic to
 $R'(n;n{-}1{-}d_1{-}\hat d,d_1,\ldots,d_k)$.
\end{proof}

\section{Concluding remarks}\label{s:conclusion}

We have proposed an asymptotic formula for the number of ways
to partition a complete graph into spanning regular subgraphs
and proved it in several cases.
The analytic method described in~\cite{mother} will be sufficient
to test the conjecture when there are many factors of high degree.
That will be the topic of a future paper.

\nicebreak

\end{document}